%% file: mcb_arxiv.tex
\documentclass[12pt, reqno, a4paper]{amsart}
\usepackage{graphicx, epsfig, psfrag}
\usepackage{amsfonts,amsmath,amssymb,amsbsy,amsthm}
\usepackage{bm}
\usepackage{color}
\usepackage{float}
\usepackage{hyperref}
\usepackage{mathrsfs}
\usepackage{subfigure}
\usepackage{bbm}
\usepackage{longtable}
\usepackage{undertilde}
\usepackage{cite}

\include{notation}

\include{notation_multi}

\newcommand{\sstl}{single-species $2$-lattice}
\newcommand{\ml}[1]{{\bm{#1}}}

\newcommand{\mlunion}{\bigcup}

\newcommand{\simm}{\mathord{\sim}}

\theoremstyle{plain}
\newtheorem{lemma}{Lemma}[section]
\newtheorem{theorem}{Theorem}[section]
\newtheorem{proposition}{Proposition}[section]

\theoremstyle{definition}
\newtheorem{definition}{Definition}[section]

\theoremstyle{remark}
\newtheorem{remark}{Remark}[section]

\def\GL{{\rm GL}}
\def\Drg{D}  
\def\OmN{\Omega_N}
\def\YN{\Y_N}

\def\wt{\bar}
\def\OmC{{\Omega}}

\title[Symmetries of $2$-lattices and the Cauchy--Born Model]{Symmetries of $2$-lattices and second order accuracy of the Cauchy--Born Model}

\author{B. Van Koten}
\address{B. Van Koten\\ 127 Vincent Hall \\ 206 Church St. SE \\
  Minneapolis \\ MN 55455 \\ USA}
\email{vankoten@math.umn.edu}

\author{C. Ortner}
\address{C. Ortner\\ Mathematics Institute \\ Zeeman Building \\
  University of Warwick \\ Coventry CV4 7AL \\ UK}
\email{christoph.ortner@warwick.ac.uk}

\thanks{BVK was supported in part by DMS-0757355, DMS-0811039,
 the PIRE Grant OISE-0967140, the University of Minnesota
 Supercomputing Institute, and the
 Department of Energy under Award Number DE-SC0002085.  CO was
supported by EPSRC grant EP/H003096 ``Analysis
 of atomistic-to-continuum coupling methods''.}

\begin{document}

\begin{abstract}
  We show that the Cauchy--Born model of a single-species $2$-lattice
  is second order if the atomistic and continuum kinematics are
  connected in a novel way.  Our proof uses a generalization to
  $2$-lattices of the point symmetry of Bravais lattices.

  Moreover, by identifying similar symmetries in multi-species pair
  interaction models, we construct a new stored energy density, using
  shift-gradients but not strain gradients, that is also second order
  accurate.

  These results can be used to develop highly accurate continuum
  models and atomistic/continuum coupling methods for materials such
  as graphene, hcp metals, and shape memory alloys.
\end{abstract}

\maketitle

\section{Introduction}
The Cauchy--Born model is a widely used continuum model for crystal
elasticity~\cite{BornHuang:1954, Ortiz:1995a, BLBL:arma2002}.
Moreover, it is a crucial ingredient in a new class of
atomistic/continuum multi-scale methods
\cite{Ortiz:1995a,XiBe:2003,Shimokawa:2004,KlZi:2006,LuOrVK:2011}.
Formal considerations and rigorous analyses have shown that the
Cauchy--Born model for Bravais lattices (simple lattices) is
second order accurate
\cite{BLBL:arma2002,Hudson:stab,E:2007a,OrTh:pre}.  By contrast, one
expects that its generalization to multi-lattices should be only
first order accurate~\cite{E:2007a}, due to the absence of point
symmetry in general multi-lattices.

In the present work, we identify two non-trivial generalizations of
the Bravais lattice point symmetry.  We show that this leads to
second order accuracy of the classical Cauchy--Born model
for a single-species $2$-lattice, {\em provided} the atomistic and
continuum kinematics are connected in a novel way. Moreover, we
identify a new stored energy density for general multi-lattices
under pair interaction model,
and we show that this energy is also second order accurate.

While these are interesting observations on their own, they have
important consequences for computational materials modeling. For
example, our results provide higher-order continuum approximations
{\em without} requiring the use of $C^1$-conforming numerical
methods. (Higher order continuum models can always be constructed
provided they are discretized using higher-order conforming numerical
methods \cite{ArGr:2005}.) Our own interest in this issue is the
application of the Cauchy--Born model in atomistic/continuum coupling
methods. At the atomistic/continuum interface, the order of accuracy
is typically lower than in the continuum bulk, due to the loss of
interaction symmetry \cite{Dobson:2008b,Ort:2011,VKL:2011,OrtZha:pre}.
By employing a second order continuum approximation, the loss can be made less severe.
For example, in the blended quasicontinuum
method~\cite{VKL:2011}, the interfacial error is controlled in terms of
an approximation parameter $k$ called the blending width.  In 1D, it is
shown in~\cite{VKL:2011} that the leading order term in the error
decreases as $k^{-\frac{3}{2}}$ when the site energy has point
symmetry (see Definition~\eqref{eq:invsymm_1lat} below).  By contrast,
we show in~\cite{OrVK:blend2} that if the site energy does not have
point symmetry, then the error decreases as $k^{-\frac{1}{2}}$. In
higher dimensions, the loss in accuracy can be so severe that the
method may cease to be consistent \cite{OrVK:blend2}.  In addition,
the symmetries that we identify allow us to investigate ghost-force
removal techniques as discussed in \cite{Shimokawa:2004, OrtZha:pre},
which cannot be employed when site energies do not possess point
symmetry.

Our framework encompasses all single-species materials with
$2$-lattice structure and general multi-lattices modeled by pure pair
interactions.  Physical systems of interest that can be described
within this framework are hcp metals (e.g., Mg, Ti, Zn), the honeycomb
lattice (graphene) or the diamond cubic structures (e.g., diamond or
silicon). Common multi-species materials with multi-lattice structure
are shape memory alloys such as Ni-Ti, Fe-Ni, Ni-Al, which are often
modelled using only pair interactions \cite{KastAck:2009,
  GhutEll:2008}.

\subsection*{Outline}
In Section \ref{sec:atm}, we introduce the atomistic energies that we
consider, and the associated notation for atomistic kinematics.
In Section \ref{sec:cb}, we derive the classical Cauchy--Born energy for
simple and multi-lattices and prove first order accuracy of the energy
in the multi-lattice case and second order accuracy in the simple
lattice case. The second order result is based on the point symmetry
of Bravais lattices, which motivates the search for similar symmetries
in multi-lattices in Section \ref{sec:symm}. We identify two
non-trivial types of point symmetries in this section. In Section
\ref{sec:2d order 2latt}, we exploit the symmetry in single-species
$2$-lattices to prove second order accuracy of the classical
Cauchy--Born energy. This result does not extend immediately to the
case of multi-lattice pair interactions. Instead, in Section
\ref{sec:pair}, we construct a novel stored energy density for
which we are again able to exploit the symmetries discovered in
Section \ref{sec:symm} to prove second order accuracy.

For the sake of simplicity, we formulate all our results for
$2$-lattices. However, it is straightforward to see that the results
for pair interactions (and only these) extend to general
multi-lattices.


\section{Atomistic models for $2$-lattices}
\label{sec:atm}

\subsection{Atomistic kinematics}
For $d \in \{1,2,3\}$, a $d$-dimensional \emph{Bravais lattice}
(simple lattice or $1$-lattice)
is a set of the form $B\Z^d$ for some strain $B \in \GL(d)$.
A \emph{$d$-dimensional $2$-lattice} is a set of
the form
\begin{equation}\label{def: complex lattice}
   \b(B\Z^d + p_0\b) \mlunion \b(B\Z^d + p_1\b),
\end{equation}
for some \emph{shifts} $p = \{p_0, p_1\} \in \R^d \times \R^d$.

We call $\Z^d$ the {\em reference lattice}
and a point $\xi \in \Z^d$ a {\em site}.
A convenient index set for a $2$-lattice is
\begin{equation*}
  \Lambda := \Z^d \times \{0,1\},
\end{equation*}
which we call the \emph{reference list}.  The reference list serves a
purpose similar to the reference domain in elasticity.  We think of
the elements of the reference list as atoms or nuclei, and we call
$(\xi, \alpha)$ the \emph{atom of index $\alpha$ at site $\xi$}.

An {\em atomistic deformation} is a map $y: \Lambda \rightarrow \Real^d$.  We
will use the notation ${y_\alpha(\xi) := y(\xi, \alpha)}$ for
evaluations of functions defined on $\Lambda$, and we call
$y_\alpha(\xi)$ the \emph{deformed position of atom $(\xi, \alpha)$}.
To make our analysis as simple as possible, we impose \emph{periodic
  boundary conditions} on the set of deformations. Fix $N \in \N$.  We
call a map ${u: \Z^d \rightarrow \Real^d}$ an \emph{$N$-periodic
  displacement} if
\begin{equation*}
u_\alpha(\xi) = u_\alpha(\xi + N\eta) \mbox{ for all } \xi, \eta \in \Z^d, \, \alpha \in \{0,1\},
\end{equation*}
and we call $y$ an \emph{$N$-periodic deformation} if for some
$N$-periodic displacement $u$ and some strain $B \in \GL (d)$ we have
\begin{equation*}
y_\alpha(\xi)  = B\xi + u_\alpha(\xi).
\end{equation*}
Throughout the remainder of the paper, we will assume that all
deformations $y$ are $N$-periodic.

\begin{remark}[Index versus species]
  It is important for our purposes to distinguish between lattices
  that are composed of identical atoms and lattices that are composed
  of atoms of two or more species.  Thus, we will draw a careful
  distinction between the species of an atom and its index.  The
  \emph{species} is the type of atom, e.g. Cu, Zn, C.  The
  \emph{index} belongs to the set $\{0,1\}$, and it tells us which of
  the component Bravais lattices making up the $2$-lattice should be
  associated with the atom.  We assume that atoms of the same index
  must be of the same species.
\end{remark}

\subsection{Atomistic energies}\label{subsec:atomistic energies}
Let $y$ be an $N$-periodic deformation, and let $\Omega_N := \{0,1,
\dots, N-1\}^d$ be a periodic cell.  For a pair interaction model, the
atomistic energy takes the form
\begin{equation*}
\Ea(y) := \sum_{\xi \in \Omega_N} \left \{ \sum_{\alpha \in \{0,1\}}
\sum_{\substack{(\eta, \beta) \in \Lambda \\ (\eta, \beta) \neq (\xi , \alpha)}}
 \smfrac{1}{2} \phi_{\alpha\beta} (|y_\beta(\eta) - y_\alpha(\xi)|) \right \},
\end{equation*}
where the functions $\phi_{\alpha\beta} : (0,\infty) \rightarrow
\Real$ are \emph{pair potentials} that may depend on the species of
interacting atoms. If the material is composed of atoms of a single
species, then there should be no dependence of the potentials on
$\alpha$ and $\beta$. Examples of multi-lattice pair interaction
models describing interesting mechanics such as the shape-memory
effect are described in \cite{KastAck:2009, GhutEll:2008}

We call the inner sum
\begin{equation}
  \label{eq:site_en_pair}
\sum_{\alpha \in \{0,1\}}
 \sum_{\substack{(\eta, \beta) \in \Lambda \\ (\eta, \beta) \neq (\xi , \alpha)}}
 \smfrac{1}{2} \phi_{\alpha\beta} (|y_\beta(\eta) - y_\alpha(\xi)|)
\end{equation}
the \emph{site-energy at $\xi$}.
To simplify our analysis,
we will assume that $\phi_{\alpha \beta} \in C^3([0,\infty); \Real)$,
even though physically realistic potentials (such as the Lennard-Jones potential)
may have singularities at $0$.
We also assume that for some
\emph{cut-off radius} $r_c \in (0, \infty)$,
\begin{equation*}
\phi_{\alpha \beta} (r) = 0 \mbox{ whenever } r > r_c.
\end{equation*}
Thus, the sum defining the site-energy is finite as long as
$y(\Lambda)$ does not have an accumulation point. To avoid discussing
this purely technical point, we assume in the following that the
interaction range is finite in the reference configuration.

In our analysis, we will also allow a more general class of potentials
than pair interactions.  We only require that the total energy can be
decomposed into a sum of localised site-energies. To that end, let
\begin{equation*}
\Rg \subset \mathbb{Z}^d \times \{0,1\} \times \{0,1\}
\setminus \b\{(0; 0, 0), (0; 1,1) \b\}
\end{equation*}
be a finite {\em interaction range}.  Throughout the remainder of the
paper, $\ml{\rho} = (\rho; \alpha, \beta)$ will denote an element of
$\Rg$.  Given $\ml{\rho} \in \Rg$ and a deformation $y \in \Y$, we
define the {\em $2$-lattice finite difference}
\begin{equation*}
D_\ml{\rho} y(\xi) := y_\beta(\xi + \rho) - y_\alpha(\xi),
\end{equation*}
and the $\Rg$-tuple $D_\Rg y(\xi) := (D_\ml{\rho} y(\xi))_{\ml{\rho} \in \Rg}$.

We assume that the total atomistic energy takes the form
\begin{equation}\label{eq:general form of atomistic energy}
\Ea(y) = \sum_{\xi \in \Omega_N} V(D_\Rg y (\xi)),
\end{equation}
where $V: (\Real^d)^\Rg \rightarrow \Real$ is a \emph{site potential}.
In our analysis, we will assume that ${V \in C^3((\Real^d)^\Rg ;
  \Real)}$. Clearly, the pair interaction site-energy
\eqref{eq:site_en_pair} is of this form.

Moreover, typical EAM potentials for hcp metals \cite{BasJoh:1994} or
bond-angle and bond-order potentials for carbon structures (graphene)
\cite{Te:1988,Br:1990} can be written in this
form. However, it may be impossible to express potentials arising
directly from quantum mechanics or electronic structure models as
in~\eqref{eq:general form of atomistic energy}.

\begin{remark}
  Fixing the interaction range $\Rg$ in the {\em reference domain}, is
  justified by the fact that we consider only elastic effects in this
  paper.
\end{remark}

\section{The Cauchy--Born energy}
\label{sec:cb}
The Cauchy--Born energy is an elastic energy which provides a good
approximation of~\eqref{eq:general form of atomistic energy} for
deformations which are close to homogeneous (i.e., smooth)
\cite{BLBL:arma2002, E:2007a, OrTh:pre}.

\subsection{Continuum kinematics}
In continuum models for $2$-lattices,
the kinematic variables are
a \emph{deformation field} $Y \in C^1(\Real^d; \Real^d)$
and a \emph{shift field} $P \in C^0(\Real^d; \Real^d)$.
We say that a pair of fields $(Y,P)$ is $N$-periodic if, for some
macroscopic strain $B \in \GL(d)$,
\begin{equation*}
  P(x+N\eta) = P(x) \quad \text{and} \quad
  y(x+N\eta) = NB\eta + y(x) \quad \mbox{for all } \eta \in \Z^d.
\end{equation*}

\subsection{The Cauchy--Born energy for Bravais lattices}
We first review the Cauchy--Born approximation for Bravais lattices.
In this case, we may ignore the shifts,
and hence atomistic deformations are now maps from $\Z^d$ to $\R^d$,
and the interaction range $\Rg$ is a subset of $\Z^d$.
The continuum kinematic variable is just a single
$N$-periodic deformation field $Y \in C^1(\Real^d; \Real^d)$.

Set $\Omega := [0,1)^d$.  We observe that $N\Omega$ is a periodic cell
for an $N$-periodic continuum deformation.  The \emph{Cauchy--Born
  energy} (for Bravais lattices) takes the form
\begin{equation}\label{eq:conventional cauchy born}
  \Ec(Y) = \int_{N\Omega} W(\D Y) \, dx,
\end{equation}
where $W: \Real^{d \times d} \rightarrow \Real \cup \{+ \infty\}$ is
the \emph{Cauchy--Born strain energy density}. For $F \in \GL(d)$,
$W(F)$ is defined to be the atomistic energy per unit volume in the
lattice $F\Z^d$.  That is, for the atomistic deformation
$y^F(\xi) := F \xi$,
\begin{equation}\label{eq:cb density simple lattice}
  W(F) := \lim_{N \to\infty} N^{-d} \sum_{\xi \in \Omega_N} V\b(D_\Rg y^F(\xi)\b) = V\b(D_\Rg y^F (0) \b),
\end{equation}
since $D_\Rg y^F(\xi) = D_\Rg y^F(0) = \b\{F\rho \b\}_{\rho \in \Rg}$ for all $\xi \in \Z^d$.

Let $Y$ be a continuum deformation.
We observe that
\begin{equation}
  \label{eq:1latt_cb_dirder}
  W(\D Y(x)) = V(\D_\Rg
  Y(x)), \quad \text{and} \quad
  \Ec(Y) = \int_{N\OmC} V(\D_\Rg Y) \dx,
\end{equation}
where $\D_\Rg Y(x) := \b\{\D_\rho Y(x)\b\}_{\rho\in\Rg}$.

\subsection{The Cauchy--Born energy for $2$-lattices}
We now explain how the Cauchy--Born model is traditionally generalized
to \mbox{$2$-lattices}.
Let $Y$ and $P$ be $N$-periodic deformation and shift fields.
The \emph{Cauchy--Born energy} (for 2-lattices) takes the form
\begin{equation*}
  \Ec(Y, P) := \int_{N\OmC} W(\D Y, P) \dx,
\end{equation*}
where $W: \Real^{d \times d} \times (\Real^d)^n \rightarrow \Real \cup \{ + \infty\}$
is the \emph{Cauchy--Born strain energy density for $2$-lattices.}

As in the case of Bravais lattices, $W(F,P)$ is the energy per unit
volume in a lattice subjected to a homogeneous deformation with
strain $F$ and shift $P$.
That is, for the deformation $y^{F, P}$ defined by
$y^{F, P}_\alpha(\xi) := F\xi + \alpha P$, we have
\begin{equation}
  \label{eq:first formula for cb complex}
  W(F, P) := \lim_{N \rightarrow \infty} N^{-d} \sum_{\xi \in \Omega_N} V\b(D_\Rg y^{F,P} (\xi)\b)  =
  V\b(D_\Rg y^{F, P}(0) \b),
\end{equation}
since $D_\Rg y^{F,P} (\xi) = D_\Rg y^{F,P} (0)
= \b\{ F\rho + (\beta - \alpha) P \b\}_{\ml{\rho} \in \Rg}$
for all $\xi \in \Z^d$.
(Recall the convention $\ml{\rho} =(\rho; \alpha, \beta)$.)

We now generalize \eqref{eq:1latt_cb_dirder}. For $\ml{\rho} \in \Rg$,
 we define the \emph{$2$-lattice directional derivative} $\D_\ml{\rho}$ by
\begin{equation}
  \D_{\ml{\rho}} (Y,P)(x) := \D_\rho Y (x) + (\beta - \alpha)P(x),
\label{eq:2latt_dirder}
\end{equation}
and we set ${\D_\Rg (Y,P)(x)} := {\b\{\D_\ml{\rho} (Y,P)
  (x)\b\}_{\ml{\rho} \in \Rg}}$.  Thus, by~\eqref{eq:first formula for
  cb complex},
\begin{align}
W(\D Y, P) =& V(\D_\Rg (Y,P)(x)), \mbox{\quad and} \nonumber \\
\Ec(Y, P) =& \int_{N \OmC} V\b( \D_\Rg (Y,P)(x)\b) \dx.
\label{eq:Ec_multilatt_dirder}
\end{align}

\subsection{Error estimates for the Cauchy--Born energy}
\label{sec:err_energy}
In order to compare the atomistic and Cauchy--Born energies,
we must specify how to generate an atomistic deformation $y$
from a continuum deformation field $Y$ and shift field $P$. In this
section, we adopt the classical approach, however, we will
see in Section \ref{sec:2d order 2latt} that for a different
identification of the atomistic and continuum variables, the
Cauchy--Born energy can be a better approximation.

The Cauchy--Born energy accurately approximates the atomistic energy
when the deformation and shift fields are ``smooth'' or ``nearly
homogenous.''  To make this precise, we define a family of
increasingly smooth deformations $Y^N$ and shifts $P^N$.  Let $Y$ and
$P$ be fixed, $1$-periodic deformation and shift fields.  Define the
\emph{scaled} deformation and shift fields $Y^N$ and $P^N$ by
\begin{equation}
Y^N(x) := NY \left( \frac{x}{N} \right ) \mbox{\quad and \quad}
P^N(x) := P \left( \frac{x}{N} \right ),
\end{equation}
and define a corresponding atomistic deformation $y^N$ by
\begin{equation}
y^N_0 (\xi) := Y^N(\xi) \mbox{\quad and \quad}
y^N_1(\xi) := Y^N(\xi) + P^N(\xi).
\label{eq:definition of yN}
\end{equation}

Observe that $Y^N$, $P^N$, and $y^N$ are all $N$-periodic.
We adopt the convention that the atomistic energy of an $N$-periodic
deformation is always the energy of the periodic cell $\Omega_N$,
and that the Cauchy--Born energy is the energy of $N\Omega$;
that is,
\begin{align*}
\Ea\b(y^N\b) :=& \sum_{\xi \in \Omega_N} V\b(D_\Rg y^N(\xi)\b),
\mbox{\quad and } \\
 \Ec\b (Y^N, P^N \b) :=& \int_{N \Omega} W\b(\D Y^N, P^N \b) \, \dx.
\end{align*}

Since we have $\D Y^N(x) = \D Y\left(\frac{x}{N}\right)$,
the strain energy density satisfies
\begin{equation*}
W \b(\D Y^N(x), P^N(x) \b) = W \left (\D Y \left ( \frac{x}{N}\right), P\left(\frac{x}{N}\right)\right),
\end{equation*}
and so the Cauchy--Born energy has the scaling invariance
\begin{equation*}
N^{-d} \Ec\b(Y^N, P^N\b) = \Ec(Y,P).
\end{equation*}
This suggests that we should treat $\Ec(Y,P)$ as an approximation of
the atomistic energy per atom $N^{-d} \Ea\b(y^N\b)$. Indeed we show in
Proposition \ref{prop:convergence of cb nlattice} that
\begin{equation*}
  \lim_{N \rightarrow \infty} N^{-d}\Ea\b(y^N\b) = \Ec(Y,P).
\end{equation*}
That is, $\Ec(Y, P)$ may be understood as the {\em continuum limit} of
$N^{-d} \Ea\b(y^N\b)$ as $N \to \infty$; see~\cite{BLBL:arma2002}.
The proof is elementary, but we provide it nevertheless for comparison with our subsequent analysis.

\begin{proposition}[Convergence of Cauchy--Born energy for $2$-lattices]
  \label{prop:convergence of cb nlattice}
  Let $(Y,P)$ be a $1$-periodic deformation and shift field.
  Let $y^N$ be the scaled atomistic deformation
  defined by~\eqref{eq:definition of yN}.
  Then we have
  \begin{equation*}
    \b|N^{-d} \Ea \b(y^N\b) - \Ec(Y,P) \b| \leq  C N^{-1}
    \left\{ \|\D^2 Y\|_{\infty}
    + \|\D P \|_{\infty} \right\},
  \end{equation*}
  where the constant $C$ is a function of $\Rg$ and $\|V\|_{C^1}$.
\end{proposition}

\begin{proof}

Since the midpoint rule is exact for constant functions,
a standard quadrature estimate gives
\begin{align*}
N^{-d} \Ec\b(Y^N, P^N\b) 
&=  N^{-d} \int_{N\Omega} V\b(\D_\Rg \b(Y^N,P^N\b) \b) \, \dx \nonumber \\
&= N^{-d} \sum_{\xi \in \Omega_N} V\b(\D_\Rg \b(Y^N,P^N\b) \b) (\xi) \nonumber \\
&\quad +O\left(\|V\|_{C^1} \b\{\b\|\D^2 Y^N \b\|_\infty + \b\|\D P^N \b\|_\infty \b\} \right).
\end{align*}
Thus,
\begin{align}
Err(N) :=& N^{-d}\b \{ \Ea\b(y^N\b) - \Ec \b(Y^N,P^N \b) \b\} \nonumber \\
=& N^{-d}\sum_{\xi \in \Omega_N} V\b(D_\Rg y^N(\xi)\b)
- V\b(\D_\Rg \b(Y^N,P^N\b) \b) (\xi), \nonumber \\
&\quad+O\left(\|V\|_{C^1} \b\{\b\|\D^2 Y^N \b\|_\infty + \b\|\D P^N \b\|_\infty \b\} \right).
\label{eq:convergence 2lat 1}
\end{align}
By the mean value theorem,
\begin{align}
&\left |V \b(D_\Rg y^N(\xi)\b) - V \b( \D_\Rg \b(Y^N,P^N\b)(\xi) \b) \right | \nonumber \\
&\qquad \qquad \qquad \qquad \leq \sum_{\ml{\rho} \in \Rg}
\|V\|_{C^1} \b|D_\ml{\rho} y^N(\xi) - \D_\ml{\rho} \b(Y^N,P^N\b)(\xi) \b|,
\label{eq:convergence 2lat 2}
\end{align}
and using Taylor's theorem,
\begin{align}
D_\ml{\rho} y^N(\xi) - \D_\ml{\rho} \b(Y^N,P^N\b)(\xi)
&= y^N_\beta(\xi+\rho) - y^N_\alpha(\xi) \nonumber \\
&\quad - \D_\rho Y^N(\xi) +(\alpha - \beta) P^N(\xi)
\nonumber \\
&= Y(\xi + \rho) - Y(\xi) - \D_\rho Y^N(\xi) \nonumber \\
&\quad + \beta \b\{ P^N(\xi +\rho) -  P^N(\xi) \b\}
\nonumber \\
&= O\left(\b\|\D^2 Y^N \b\|_\infty + \b\|\D P^N \b\|_\infty \right).
\label{eq:convergence 2lat 3}
\end{align}

Combining~\eqref{eq:convergence 2lat 1}, \eqref{eq:convergence 2lat 2},
and~\eqref{eq:convergence 2lat 3} shows
\begin{equation*}
N^{-d}|\Ea(y^N) - \Ec(Y^N, P^N)|
\leq C \b\{\b\|\D^2 Y^N \b\|_\infty + \b\|\D P^N \b\|_\infty \b\}
\end{equation*}
where the constant $C$ is a function of $\| V\|_{C^1}$ and $\Rg$.
We now observe
\begin{equation*}
\b\|\D^2 Y^N \b\|_\infty = N^{-1} \|\D^2 Y\|_\infty
\mbox{\quad and \quad} \b\|\D P^N \b\|_\infty = N^{-1}\|\D P\|_\infty,
\end{equation*}
and the result follows.
\end{proof}

For Bravais lattices, the estimate given in
Proposition~\ref{prop:convergence of cb nlattice} can be improved if
we assume that the site potential $V$ has \emph{point
  symmetry}~\cite{Hudson:stab}:
\begin{equation}\label{eq:invsymm_1lat}
V(\ml{g}) = V(-\{g_{-\rho}\}_{\rho \in \Rg}) \mbox{ for all } \ml{g} \in (\Real^d)^\Rg.
\end{equation}
We show in Lemma~\ref{lem:inversion symmetry Bravais lattice} that
under physically reasonable assumptions, one can always take $V$ to be
point symmetric.  Let $R:\Real^d \rightarrow \Real^d$ defined by $Rx
=-x$ be \emph{point reflection in the origin}.
The first assumption is that
\begin{equation*}
  \Ea(y) = \Ea (y \circ R) \mbox{ for all } y \in \Y.
\end{equation*}
Physically, this assumption is motivated by the observation that
permutation or relabelling of the atoms does not change the energy,
and that Bravais lattices are invariant under the point inversion $R$,
which means that $R$ provides such a relabelling.

The second assumption is that
\begin{equation*}
\Ea(y) = \Ea(-y) \mbox{ for all } y \in \Y,
\end{equation*}
which is motivated by the principle that the energy should be
unchanged if the configuration of the atoms is reflected.

\begin{remark}
  Permutation invariance can be expected to hold whenever there is
  only one species of atoms.  If there is more than one species, then
  we can only assume that the energy is invariant under those
  permutations which preserve the species of the atoms. Hence, Lemma
  \ref{lem:inversion symmetry Bravais lattice} cannot be immediately
  applied to general multi-lattices.
\end{remark}

\begin{lemma}\label{lem:inversion symmetry Bravais lattice}
Assume that for all deformations $y$,
\begin{equation}
\label{eq:point sym total energy simple}
\Ea(y) = \Ea(- y \circ R),
\end{equation}
where $Rx = -x$; then
\begin{displaymath}
\Ea(y) = \sum_{\xi \in \OmN} \wt{V}\b(D_{\wt{\Rg}} y(\xi) \b),
\end{displaymath}
where $\wt{\Rg} = \Rg \cup -\Rg$,
and $\wt{V} : (\R^d)^{\wt{\Rg}} \to \R$ defined by
\begin{displaymath}
    \wt{V}\b( \{ g_{\ml\rho} \}_{\ml{\rho} \in \wt{\Rg}} \b)
    := \smfrac{1}{2} V\b( \{ g_{\ml\rho} \}_{\ml\rho\in \Rg} \b)
    + \smfrac{1}{2} V\b( \{- g_{- \ml\rho} \}_{\ml\rho \in -\Rg} \b),
\end{displaymath}
is point symmetric~\eqref{eq:invsymm_1lat}.
\end{lemma}

\begin{proof}
  First, we note that, if $y$ is an $N$-periodic deformation then so
  is $-y \circ R$, since for all $\eta \in \Z^d$
\begin{equation*}
-y \circ R (\xi + N\eta) = -y(- \xi - N \eta ) = - y(-\xi) = - y \circ R (\xi).
\end{equation*}
By~\eqref{eq:point sym total energy simple}, we have
  \begin{align*}
    \Ea (y) &= \smfrac{1}{2} \Ea(y) + \smfrac{1}{2} \Ea( - y \circ R) \\
    &= \frac12 \sum_{\xi \in \OmN} V\b(D_\Rg y(\xi)\b) +
    \frac12 \sum_{\xi \in \OmN} V\b(- D_\Rg (y \circ R) (\xi)\b) \\
    &= \frac12 \sum_{\xi \in \OmN} V(D_\Rg y(\xi)) +
    \frac12 \sum_{\xi \in \OmN} V(- D_{-\Rg} y (- \xi)).
  \end{align*}
  The last equality follows since for $\rho \in \Rg$ we have
  \begin{align*}
    D_\rho y \circ R (\xi) &= y \circ R (\xi + \rho) - y \circ R(\xi) \\
    &= y(-\xi - \rho) - y(-\xi) \\
    &= D_{-\rho} y(-\xi).
  \end{align*}
  Employing periodicity of $y$, we can shift the second summation over
  $-\OmN$ back to $\OmN$ and upon relabelling, obtain
  \begin{align*}
    \Ea (y)
    &= \frac{1}{2}\sum_{\xi \in \OmN} V(\Drg y(\xi)) + \frac{1}{2}
    \sum_{\xi \in \OmN} V(- D_{-\Rg} y (\xi))  \\
    &= \sum_{\xi \in \OmN} \wt{V}(\wt{D}_{\wt{\Rg}} y(\xi)).
  \end{align*}
  Finally, point symmetry of $\wt\Rg$ and of $\wt{V}$ are obvious.
\end{proof}

Point symmetry of $V$ implies a symmetry of the partial derivatives of
$V$ in homogeneous states.  For $\rho \in \Rg$ and $\ml{g} \in
(\Real^d)^\Rg$, we define $V_\rho(\ml{g}) := \frac{\partial
  V}{\partial g_\rho} (\ml{g})$.  If $V$ is point symmetric, then we
have
\begin{equation}
\label{eq:symm of deriv 1lattice}
V_\rho (F\cdot\Rg) = - V_{-\rho} (F \cdot \Rg) \quad
\mbox{ for all } F \in \GL(d).
\end{equation}
To see this, one differentiates the identity \eqref{eq:invsymm_1lat}
with respect to $g_\rho$ and then evaluates it at $\ml{g} = F \cdot
\Rg$.

We can now prove the that the Cauchy--Born approximation for Bravais
lattices is second order accurate, which was previously observed in
\cite{BLBL:arma2002, Hudson:stab, MaSu:2011}. We nevertheless give a
complete proof of the result, since it motivates our subsequent
analysis for $2$-lattices.  For Bravais lattices, there is no reason
to make a distinction between the atomistic and continuum variables
since we will not consider multiple ways of relating the two.  Thus,
given a fixed $1$-periodic deformation $Y:\Real^d \rightarrow
\Real^d$, we let
\begin{equation}
  \label{eq:defn_YN_Bravais}
  Y^N(x) := NY \left ( \frac{x}{N} \right ),
\end{equation}
and we interpret $Y^N$ as both an atomistic and a continuum deformation.

\begin{proposition}[Convergence for Bravais lattices.]
\label{prop:error for simple cb}
Let $Y \in C^3(\Real^d;\Real^d)$ be a $1$-periodic deformation, and
let $Y^N$ be defined by \eqref{eq:defn_YN_Bravais}; then,
\begin{equation*}
  \b| N^{-d} \Ea\b(Y^N\b) - \Ec(Y) \b| \leq C N^{-2} \b \{\|\D^2 Y\|^2_\infty
+ \|\D^3 Y\|_\infty \b \},
\end{equation*}
where the constant $C$ is a function of $\Rg$ and $\|V\|_{C^2}$.
\end{proposition}
\begin{proof}
Since the midpoint rule is exact for affine functions, a standard
quadrature estimate gives
\begin{align}
N^{-d} \Ec\b(Y^N\b) &= N^{-d} \int_{N\Omega} V \b(\D_\Rg Y^N \b) \, \dx \nonumber \\
&= N^{-d}\sum_{\xi \in \Omega_N} V\b(\D_\Rg Y^N(\xi)\b) \nonumber \\
&\quad + O \left ( \b\|\D^2 V\b(\D_\Rg Y^N\b) \b\|_\infty \right ).
\label{eq:convergence 1lat 1}
\end{align}
Set $V_\rho(\ml{g}) := \frac{\partial}{\partial g_\rho} V(\ml{g})$,
and let $V_\rho(\xi) := V_\rho\b(\D_\Rg Y^N (\xi) \b)$.
By Taylor's theorem,
\begin{equation}
\begin{split}
V\b(D_\Rg Y^N(\xi)\b) - V\b(\D_\Rg Y^N(\xi)\b) &=
\sum_{\rho \in \Rg} V_\rho (\xi)\cdot \b\{D_\rho Y^N(\xi) - \D_\rho Y^N(\xi) \b\} \\
&+ O(\|V\|_{C^2} \b|D_\rho Y^N (\xi) - \D_\rho Y^N(\xi) \b |^2),
\end{split}
\label{eq:convergence 1lat 2}
\end{equation}
and
\begin{align}
D_\rho Y^N(\xi) - \D_\rho Y^N(\xi)
&= Y^N(\xi + \rho) - Y^N (\xi) - \D_\rho Y^N(\xi) \nonumber \\
&= \frac{1}{2} \D^2_\rho Y^N (\xi) + O\b(\b\|\D^3 Y^N\b\|_\infty \b).
\label{eq:convergence 1lat 3}
\end{align}
Combining~\eqref{eq:convergence 1lat 2} and~\eqref{eq:convergence 1lat 3}, we have
\begin{equation}
\begin{split}
V\b(D_\Rg Y^N(\xi)\b) - V\b(\D_\Rg Y^N\b) &=
\frac{1}{2}\sum_{\rho \in \Rg} V_\rho (\xi) \cdot  \D_\rho^2 Y^N(\xi) \\
&\quad
+ O\left(\b\{ \b\|\D^2 Y^N\|^2_\infty + \b\|\D^3 Y^N\|_\infty \b\} \right)
\label{eq:convergence 1lat 4}
\end{split}
\end{equation}
By symmetry of the partial derivatives of $V$~\eqref{eq:symm of deriv 1lattice},
the first term on the right hand side of~\eqref{eq:convergence 1lat 4} vanishes.
We have
\begin{align*}
  \sum_{\rho\in\Rg} V_\rho \cdot \D_\rho^2 Y^N &= \frac12
  \sum_{\rho\in\Rg} \B( V_\rho \cdot \D_\rho^2 y^N + V_{-\rho} \cdot
  \D_{-\rho}^2 Y^N \B) \\
  &= \frac12   \sum_{\rho\in\Rg} \B( V_\rho \cdot \D_\rho^2 Y^N - V_{\rho} \cdot
  \D_{\rho}^2 Y^N \B) = 0,
\end{align*}
where all functions above are evaluated at $\xi \in \Omega_N$.
Therefore, combining~\eqref{eq:convergence 1lat 1},\eqref{eq:convergence 1lat 2},
and~\eqref{eq:convergence 1lat 4} gives
\begin{equation*}
N^{-d} \b| \Ea\b(Y^N\b) - \Ec\b(Y^N\b) \b|
 \leq C \b\{ \b\|\D^2 Y^N \b \|^2_\infty + \b\|\D^3 Y^N \b \|_\infty \b\},
\end{equation*}
where the constant $C$ depends on $\Rg$ and $\|V\|_{C^2}$.
Finally, we observe that
\begin{equation*}
\b\|\D^2 Y^N \b \|^2_\infty = N^{-2} \b\|\D^2 Y \b \|^2_\infty
\mbox{\quad and \quad} \b\|\D^3 Y^N \b \|_\infty = N^{-2}\b\|\D^3 Y \b \|_\infty,
\end{equation*}
and the result follows.
\end{proof}

\section{Symmetries of the $2$-lattice site energy}
\label{sec:symm}
We showed in the previous section how the point symmetry
\eqref{eq:invsymm_1lat} of Bravais lattices implies second order accuracy of
the Cauchy--Born energy in the continuum limit.
At first glance, $2$-lattices do not possess a point symmetry: e.g.,
inversion of the honeycomb lattice about a lattice point yields a
shifted honeycomb lattice; see Figure \ref{fig:sym_carbon}. Note,
however, that inverting the lattice about the center of an edge leaves
it invariant.

\begin{figure}[t]
  \includegraphics[width=7cm]{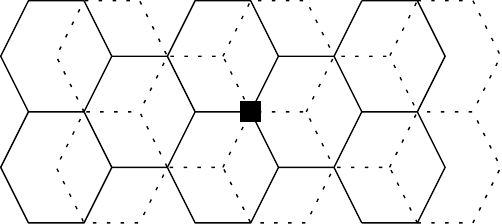}
  \caption{\label{fig:sym_carbon} Point reflection of the honeycomb
    lattice about a lattice site yields a shifted honeycomb lattice.}
\end{figure}

The purpose of this section is to explore how this and other
symmetries may be exploited to extend~\eqref{eq:invsymm_1lat}.  We
will identify two such non-trivial extensions: for the case of a
single species $2$-lattice with a general site energy, and for the
case of a multi-lattices with pair interaction energy (however, we
will only formulate the result for $2$-lattices). In Sections
\ref{sec:2d order 2latt} and \ref{sec:pair}, we will then exploit these
observations to derive ``simple'' second order continuum models.

\subsection{Symmetries for single-species $2$-lattices}
\label{sec:symm_1}
\begin{figure}[t]
  \includegraphics[width=12cm]{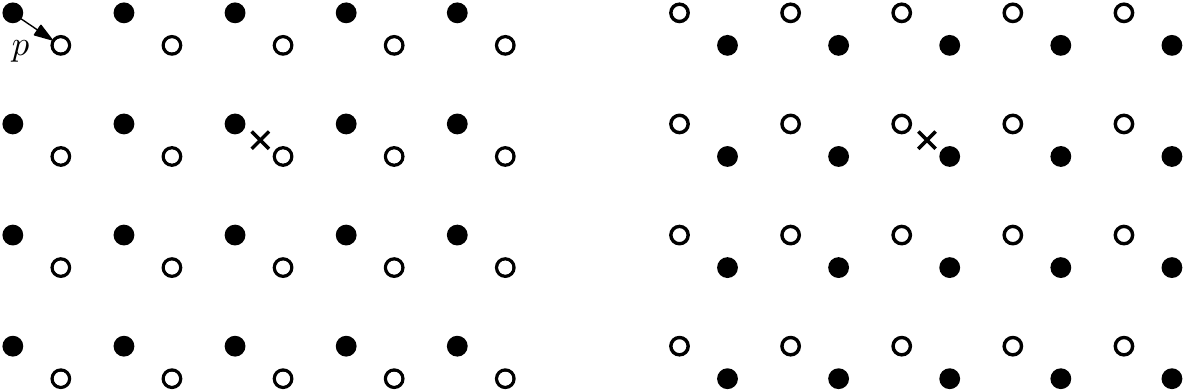}
  \caption{\label{fig:symm_2latt} Point reflection of a 2-lattice
    about the ``centroid'' ${\sf x}$ of a lattice site,
    containing $\{\circ, \bullet\}$. The lattice remains invariant, however
    the atom indices are reversed. If the atoms are of the same
    species, then the physical configuration remains invariant.}
\end{figure}
Any single-species $2$-lattice has a reflection symmetry about the
``centroid'' of a lattice site; see Figure \ref{fig:symm_2latt}.
Algebraically, let ${B \in \GL(d)}$, ${p \in \Real^d} \setminus B
\Z^d$, and let $\mathscr{L} := B \mathbb{Z}^d \, \mlunion \, (B
\mathbb{Z}^d + p)$ be a \sstl. Equivalently, we may shift the lattice
by $-p/2$ to redefine it as
\begin{equation*}
  \mathscr{L} := \b(B \mathbb{Z}^d - \smfrac{p}{2} \b) \, \mlunion \, \b(B
\mathbb{Z}^d + \smfrac{p}{2}\b),
\end{equation*}
which immediately reveals the symmetry
\begin{equation}
  \label{eq:2latt_inv_symm}
  - \mathscr{L} = \mathscr{L}
\end{equation}
Note, however, that if $\mathscr{L}$ has two species, then this
operation reverses the species and is therefore {\em not} a symmetry
of a $2$-species $2$-lattice.

The symmetry \eqref{eq:invsymm_1lat} can also be thought of as a
permutation of the reference list: $R(\rho; \alpha) := (- \rho; \neg
\alpha)$ where $\neg 0 := 1$ and $\neg 1 := 0$, then $R\Lambda =
\Lambda$. We can also define an analogous operation on interactions
$\neg : \mathbb{Z}^d \times \{0,1\} \times \{0,1\} \rightarrow
\mathbb{Z}^d \times \{0,1\} \times \{0,1\}$,
\begin{equation}
  \label{eq:defn_neg}
  \neg (\rho; \alpha, \beta)  := (-\rho;\neg \alpha, \neg \beta).
\end{equation}
The following Proposition provides a $2$-lattice analogy to the
observation that ${D_\rho y^F = -D_{-\rho}y^F}$ in
Bravais lattices, which was a key step to obtain second order
accuracy of the Bravais lattice Cauchy--Born energy.

\begin{proposition}
  \label{th:symm_1:refl}
  Let $F \in \GL(d)$ and $p  = (p_0, p_1) \in \R^{2d}$, then
  \begin{equation*}
    D_{\ml{\rho}} y^{F,p} = - D_{\neg \ml{\rho}} y^{F,p}
    \qquad \forall \ml{\rho} \in \Rg.
  \end{equation*}
\end{proposition}
\begin{proof}
  Let $y = y^{F,p}$ and $\ml{\rho} = (\rho; \alpha, \beta) \in \Rg$,
  then
  \begin{align*}
    - D_{\neg \ml{\rho}} y(0) =\,& - \b( y_{\neg\beta}(-\rho) -
    y_{\neg\alpha}(0) \b)\\
    =\,& - \b( F (-\rho) + p_{\neg\beta} \b) + \b( p_{\neg\alpha} \b)
    \\
    =\,& F \rho + p_{\neg\alpha} - p_{\neg\beta}
  \end{align*}
  One readily checks by case distinction that $p_{\neg\alpha} -
  p_{\neg\beta} = p_\beta - p_\alpha$, which concludes the proof.
\end{proof}

\begin{definition}[Point symmetry for a single-species
  $2$-lattice]
  We say that the site potential $V$ is \emph{point symmetric
    (with respect to $\neg$)} if $\Rg = \neg \Rg$ and
  \begin{equation}
    \label{eq:defn_invsymm_1}
    V\b( \{ g_{\ml{\rho}} \}_{\ml{\rho} \in \Rg}\b)
    = V\b(- \{ g_{\neg \ml{\rho}} \}_{\ml{\rho} \in \Rg} \b) \qquad
    \forall \{ g_{\ml{\rho}} \}_{\ml{\rho} \in \Rg} \in (\R^d)^\Rg.
  \end{equation}
\end{definition}

We show next that, under physically realistic assumptions (invariance
of the energy under permutations and isometries) it is always possible
to define a point symmetric site potential for a \sstl.

\begin{proposition} \label{lem:symmetrization for sstl} Suppose that
  $\Ea$ given by \eqref{eq:general form of atomistic energy} is
  invariant under the permutation $R$ of the index list, and under the
  isometry $x \mapsto -x$:
  \begin{equation}    \label{eq:isotropy of total energy sstl}
    \Ea (y) = \Ea(y \circ R) = \Ea(-y\circ R);
  \end{equation}
  then,
  \begin{displaymath}
    \Ea(y) = \sum_{\xi \in \OmN} \wt{V}\b( D_{\wt{\Rg}} y(\xi) \b),
  \end{displaymath}
  where $\wt{\Rg} = \Rg \cup \neg\Rg$, and $\wt{V} : (\R^d)^{\wt{\Rg}}
  \to \bar\R$,
  \begin{displaymath}
    \wt{V}\b( \{ g_{\ml\rho} \}_{\ml{\rho} \in \wt{\Rg}} \b)
    := \smfrac{1}{2} V\b( \{ g_{\ml\rho} \}_{\ml\rho\in \Rg} \b)
    + \smfrac{1}{2} V\b( \{- g_{\neg \ml\rho} \}_{\ml\rho \in \neg\Rg} \b),
  \end{displaymath}
  is point symmetric \eqref{eq:defn_invsymm_1}.
\end{proposition}
\begin{proof}
  Fix $y \in \YN$ and let $\wt{\Rg}, \wt{V}$ be defined as above. We
  first show that $y' := - y \circ R \in \YN$. To that end we must
  show that the associated displacement $u' := y' - y^{B,0}$ is
  $N$-periodic. To see this, let $u := y - y^{B,0}$, then $u$ is
  $N$-periodic and we have
  \begin{align*}
    u_\alpha'(\xi)
    &= - (y\circ R)_\alpha(\xi) - B \xi
    = - y_{\neg\alpha}(-\xi) - B \xi
    \\
    &= - u_{\neg\alpha}(-\xi) + B \xi - B \xi = - u_{\neg\alpha}(-\xi).
  \end{align*}
  Since $u_{\neg\alpha}$ is $N$-periodic it follows that $u'_\alpha$
  is $N$-periodic.

  The rest of the proof is analogous to the proof of Proposition
  \ref{lem:inversion symmetry Bravais lattice}. Employing
  \eqref{eq:isotropy of total energy sstl}, we have
  \begin{align*}
    \Ea (y) &= \smfrac{1}{2} \Ea(y) + \smfrac{1}{2} \Ea( - y \circ R) \\
    &= \frac12 \sum_{\xi \in \OmN} V\b(D_\Rg y(\xi)\b) +
    \frac12 \sum_{\xi \in \OmN} V\b(- D_\Rg (y \circ R) (\xi)\b) \\
    &= \frac12 \sum_{\xi \in \OmN} V(D_\Rg y(\xi)) +
    \frac12 \sum_{\xi \in \OmN} V(- D_{\neg \Rg} y (- \xi)).
  \end{align*}
  The last equality follows since for $\ml{\rho} = (\rho; \alpha,
  \beta) \in \Rg$ we have
  \begin{align*}
    D_\ml{\rho} y \circ R (\xi) &= y \circ R (\xi+\rho, \beta) - y \circ R (\xi, \alpha)  \\
    &= y(-\xi - \rho, \neg \beta) - y(-\xi, \neg \alpha) \\
    &= D_{\neg \ml{\rho}} y(-\xi).
  \end{align*}

  Since $y - y^{B,0}$ is periodic, we can shift the second summation
  over $-\OmN$ back to $\OmN$ and, upon relabelling the summation
  variable, obtain
  \begin{align*}
    \Ea (y)
    &= \frac{1}{2}\sum_{\xi \in \OmN} V(\Drg y(\xi)) + \frac{1}{2}
    \sum_{\xi \in \OmN} V(- D_{\neg \Rg} y (N e - \xi))  \\
    &= \frac{1}{2}\sum_{\xi \in \OmN} V(\Drg y(\xi)) + \frac{1}{2}
    \sum_{\xi \in \OmN} V(- D_{\neg \Rg} y (\xi))   \\
    &= \sum_{\xi \in \OmN} \wt{V}(\wt{D} y(\xi)),
  \end{align*}
  where $e = (1, \dots, 1) \in \R^d$.

  Finally, point symmetry of $\wt\Rg$ and of $\wt{V}$ are obvious.
\end{proof}

\subsection{Symmetry of pair interaction energies}
Suppose that $\Ea$ is given by the pair interaction model
\eqref{eq:site_en_pair} for a $2$-lattice with two different species;
that is, the site energy $V:(\Real^d)^\Rg \rightarrow \Real$ is
defined by
\begin{equation}\label{def:pair site potential}
V(\ml{g}) = \frac{1}{2} \sum_{\ml{\rho} \in \Rg} \phi_{\alpha \beta} \b(|D_\ml{\rho} y(\xi) |\b).
\end{equation}
(Recall the convention that $\ml{\rho} = (\rho, \alpha, \beta)$.)  We
will now identify a symmetry in this model
that will play a role analogous to the operation $\neg$ in the previous section.

The key observation in this case is that, physically, we should
require $\phi_{\alpha\beta} = \phi_{\beta\alpha}$, which motivates the
operation $\simm : \mathbb{Z}^d \times \{0,1\} \times \{0,1\} \rightarrow
\mathbb{Z}^d \times \{0,1\} \times \{0,1\}$ by
\begin{equation}
  \label{eq:defn_sim}
  \simm (\rho; \alpha, \beta) := (-\rho; \beta, \alpha).
\end{equation}
We immediately obtain the following result.

\begin{proposition}[Symmetry of pair interaction site potentials]\label{lem:symm of pair site potential}
  Let the site energy $V$ be defined by \eqref{def:pair site
    potential} and suppose that $\Rg = \simm \Rg$ and
  $\phi_{\alpha\beta} = \phi_{\beta\alpha}$ for all $\alpha,\beta \in
  \{0,1\}$; then $V$ satisfies the point symmetry
  \begin{equation}
    \label{eq:invsymm_pair}
    V\b(\{g_{\ml{\rho}}\}_{\ml\rho\in\Rg}\b)
    = V\b(\{- g_{\simm\ml{\rho}}\}_{\ml\rho \in \Rg}\b)  \qquad
    \forall \{g_{\ml\rho}\}_{\ml\rho\in\R} \in (\R^d)^{\Rg}. 
  \end{equation}
\end{proposition}
\begin{proof}
  Under the stated assumptions,
  \begin{align*}
    V\b(\{-g_{\simm\ml{\rho}}\}_{\ml\rho\in\Rg}\b)
    &= \sum_{\ml\rho\in\Rg}
    \phi_{\alpha\beta}(|-g_{(-\rho;\beta,\alpha)}|) \\
    &= \sum_{\ml\rho\in\simm\Rg}
    \phi_{\beta\alpha}(|g_{(-\rho;\beta,\alpha)}|).
  \end{align*}
  Relabelling $\rho' = -\rho, \alpha' = \beta, \beta' = \alpha$ we
  obtain \eqref{eq:invsymm_pair}.
\end{proof}

\begin{remark}
  1. Unlike the symmetry $\neg$ of a \sstl, $\simm$ does not arise from
  any globally defined permutation of the atoms of the material.  For
  this reason, one cannot expect that an arbitrary total energy will
  be symmetric under $\simm$, and therefore one cannot hope to
  symmetrize general site potentials. 

  2. The symmetry $\simm$ can be immediately applied to $n$-lattice
  pair interactions for $n > 2$. All our subsequent results generalize
  as well.
\end{remark}

\section{Second order accuracy of the Cauchy--Born Energy}
\label{sec:2d order 2latt}
In this section, we give a new rule relating the atomistic and
continuum variables, and we show that under this rule, the
Cauchy--Born energy for single-species $2$-lattices is second order
accurate.  Given $1$-periodic deformation and shift fields $Y$ and
$P$, we first rescale them to atomic units
\begin{equation*}
  Y^N(x) := N Y \left ( \frac{x}{N} \right ) \mbox{\quad and \quad}
  P^N(x) := P \left ( \frac{x}{N} \right ).
\end{equation*}
We then define a corresponding atomistic deformation $y^N$ by
\begin{equation}
\label{eq:2order atc connection}
y^N_0(\xi) := Y^N(\xi) - \frac{1}{2} P^N(\xi) \mbox{\quad and \quad}
y^N_1(\xi) := Y^N(\xi) + \frac{1}{2} P^N(\xi).
\end{equation}
When~\eqref{eq:2order atc connection} is used to connect the atomistic
and continuum fields, we interpret $Y^N(\xi)$ as the \emph{deformation
  of the centroid} of the atoms at site $\xi$; indeed, the inverse
transformation of~\eqref{eq:2order atc connection} is
\begin{equation*}
  Y^N(\xi) := \frac{1}{2} \b\{y_0^N(\xi) + y_1^N(\xi) \b\} \mbox{\quad and \quad}
  P^N(\xi) := y_1^N(\xi) - y_0^N(\xi).
\end{equation*}

As in the Bravais lattice case, we need a symmetry of the partial
derivatives of $V$ \eqref{eq:symm of deriv 1lattice} to show second
order convergence. The following lemma establishes the appropriate
generalization.

\begin{lemma}
  \label{lemma:2latt symmetry of partial derivatives}
  Suppose that $V$ is point symmetric~\eqref{eq:defn_invsymm_1} with
  respect to the permutation operator $\neg$.  Let $\D_\ml{\rho}$ be
  defined by \eqref{eq:2latt_dirder}, and let $V_\ml{\rho} (\ml{g}) :=
  \frac{\partial}{\partial g_\ml{\rho}} V(\ml{g})$.  Then, for any
  deformation and shift fields $Y$ and $P$, we have
  \begin{equation}\label{eq:2latt symmetry of partial derivatives}
      V_{\ml{\rho}} \b(\D_\Rg (Y,P) (x)\b)
      = -V_{\neg \ml{\rho}}\b(\D_\Rg (Y,P) (x)\b).
  \end{equation}
\end{lemma}
\begin{proof}
  Differentiating \eqref{eq:defn_invsymm_1} and \eqref{eq:invsymm_pair}
  with respect to $g_{\ml\rho}$ yields
  \begin{displaymath}
    V_{\ml\rho}\b( \{ g_{\ml\rho} \}_{\ml\rho\in\Rg} \b) = -
    V_{\neg\ml\rho}\b( \{ -g_{\neg\ml\rho} \}_{\ml\rho\in\Rg} \b).
  \end{displaymath}
  Evaluating at $g_{\ml\rho} = \D_{\ml\rho} (Y,P)$ and noting that
  \begin{align*}
    -\D_{\neg\ml\rho} (Y,P) &= - \D_{-\rho} Y - (\neg\beta - \neg\alpha) P \\
    &= \D_{\rho} Y - \b( (1-\beta) -(1- \alpha) \b)P \\
    & = \D_{\rho} Y + (\beta - \alpha)P \\
    &= \D_{\ml\rho} (Y,P),
  \end{align*}
  we obtain \eqref{eq:2latt symmetry of partial derivatives}.
\end{proof}

We are now in a position to prove second order accuracy of the
Cauchy--Born model.

\begin{theorem}[Second order convergence for single species
  $2$-lattices]
  \label{thm:2d order sstl}
  Let $Y \in C^3(\R^d; \R^d)$ be a $1$-periodic deformation field and
  $P \in C^2(\R^d; \R^d)$ a $1$-periodic shift field.  Let $y^N$ be
  the scaled atomistic configuration defined by
  \eqref{eq:2order atc connection}.  Then
  \begin{equation*}
    \begin{split}
      \left |N^{-d} \Ea \b(y^N\b) - \Ec \b(Y,P \b) \right | \leq
      C N^{-2} \b \{ \,& \|\D^3 Y\|_{\infty} + \|\D^2
      P\|_{\infty} \\
      & \qquad + \|\D^2 Y\|^2_{\infty} + \|\D P\|^2_{\infty} \b \}.
    \end{split}
  \end{equation*}
  where the constant $C$ is a function of $\Rg$ and
  $\|V\|_{C^2}$.
\end{theorem}

\begin{proof}
Since the midpoint rule is exact for affine functions, a standard quadrature estimate gives
\begin{align}
N^{-d} \Ec\b(Y^N, P^N) &= N^{-d} \int_{N\Omega} V\b(\D_\Rg \b(Y^N, P^N\b)\b) \nonumber \\
&= N^{-d} \sum_{\xi \in \Omega_N} V\b(\D_\Rg \b(Y^N, P^N\b)\b) \nonumber \\
&+ O \b (\b\| \D^2V\b(\D_\Rg \b(Y^N, P^N\b)\b )\b\|_\infty\b).
\label{eq:thm cb 2latt 1}
\end{align}
It is easy to see that
\begin{align*}
&\b\| \D^2V\b(\D_\Rg \b(Y^N, P^N\b)\b )\b\|_\infty \\
&\qquad \qquad
\leq C \b( \b\|\D^3 Y^N\b\|_{\infty} + \b\|\D^2 P^N\b\|_{\infty}
 + \b\|\D^2 Y^N\b\|^2_{\infty} + \b\|\D P^N\b\|^2_{\infty} \b).
\end{align*}

We now estimate the error in $V\b( \D_\Rg \b (Y^N,P^N \b)(\xi) \b)$.
By Taylor's theorem,
\begin{equation}
\begin{split}
V\b(D_\Rg y^N\b) - V\b(\D_\Rg \b(Y^N,P^N\b)\b)
&=
\sum_{\ml{\rho}\in \Rg} V_\ml{\rho}
\cdot \b\{D_{\ml\rho} y - \D_\ml{\rho} y \b\}  \\
&+O\B(\|V\|_{C^2} \b|D_{\ml\rho} y^N- \D_\ml{\rho} \b(Y^N, P^N\b)\b|^2\B),
\label{eq:thm cb 2latt 2}
\end{split}
\end{equation}
where all functions above are evaluated $\xi \in \Omega_N$,
and $V_\ml{\rho} := V_\ml{\rho}\b( \D_\Rg \b (Y^N,P^N \b)(\xi) \b)$.
By a second application of Taylor's theorem,
\begin{align}
D_{\ml\rho} y^N(\xi)- \D_\ml{\rho} \b(Y^N, P^N\b)(\xi) &= y^N_\beta(\xi + \rho) - y^N_\alpha(\xi)
\nonumber \\
&\quad - \D_\rho Y^N (\xi) - (\beta - \alpha) P^N(\xi)  \nonumber \\
&= Y^N(\xi + \rho) + \b(\beta - \smfrac12\b) P^N(\xi+\rho) \nonumber \\
&\quad - Y^N(\xi) - \b(\alpha - \smfrac12\b) P^N(\xi) \nonumber \\
&\quad - \D_\rho Y^N (\xi) - (\beta -\alpha) P^N(\xi) \nonumber \\
&= \smfrac12 \D_\rho^2 Y^N(\xi) + \b(\beta - \smfrac12 \b) \D_\rho P^N(\xi) \nonumber \\
&\quad +O \b(\b\|\D^3 Y^N \b\|_\infty + \b\|\D^2 P^N \b\|_\infty \b).
\label{eq:thm cb 2latt 3}
\end{align}
Substituting~\eqref{eq:thm cb 2latt 3} into~\eqref{eq:thm cb 2latt 2},
and assuming again that all functions are evaluated at $\xi \in \Omega_N$,
we obtain
\begin{align}
V\b(D_\Rg y^N\b) - V\b(\D_\Rg \b(Y^N,P^N\b)\b)
&= \sum_{\ml{\rho}\in \Rg} V_\ml{\rho} \cdot e (\ml{\rho})
\nonumber \\
&+O \b(\b\|\D^3 Y^N \b\|_\infty + \b\|\D^2 P^N \b\|_\infty \b),
\label{eq:thm cb 2latt 4}
\end{align}
where
\begin{equation*}
e(\ml{\rho}) := \smfrac12 \D_\rho^2 Y^N(\xi) + \b(\beta - \smfrac12 \b) \D_\rho P^N(\xi).
\end{equation*}

We now observe that
\begin{align}
e(\neg \ml{\rho}) &= \smfrac12 \D_{-\rho}^2 Y^N(\xi)
+\b(\neg \beta - \smfrac12\b) \D_{-\rho} P^N(\xi) \nonumber \\
&=\smfrac12 \D_{-\rho}^2 Y^N(\xi)+\b((1 - \beta) - \smfrac12\b) \D_{-\rho} P^N(\xi)
\nonumber \\
&=\smfrac12 \D_{\rho}^2 Y^N(\xi)+(\beta - \smfrac12) \D_{\rho} P^N(\xi)  \nonumber \\
&= e(\ml{\rho}),
\label{eq:thm cb 2latt 5}
\end{align}
from which we deduce that the the first term on the right-hand side
of~\eqref{eq:thm cb 2latt 4} vanishes:
using~\eqref{eq:thm cb 2latt 5}
and point symmetry of the derivatives~\eqref{eq:2latt symmetry of
  partial derivatives}, we obtain
\begin{align}
\sum_{\ml{\rho} \in \Rg} V_\ml{\rho}\cdot e(\ml\rho)
&= \frac{1}{2} \sum_{\ml{\rho} \in \Rg} V_\ml{\rho} \cdot e(\ml\rho) + V_\ml{\ast\rho} \cdot e(\ast\ml\rho) \nonumber \\
&=\frac{1}{2} \sum_{\ml{\rho} \in \Rg} V_\ml{\rho}
\cdot \b \{e(\ml\rho) - e(\ast\ml\rho) \b \} \nonumber \\
&= 0.
\label{eq:thm cb 2latt 5}
\end{align}
Finally, combining~\eqref{eq:thm cb 2latt 1}, \eqref{eq:thm cb 2latt 2}, \eqref{eq:thm cb 2latt 3}, ~\eqref{eq:thm cb 2latt 4}, and~\eqref{eq:thm cb 2latt 5} gives
\begin{equation*}
\begin{split}
  &\left |N^{-d} \Ea \b(y^N\b) - \Ec \b(Y,P \b) \right |  \\
  &\qquad \qquad \leq C \b( \b\|\D^3 Y^N\b\|_{\infty} + \b\|\D^2 P^N\b\|_{\infty}
 + \b\|\D^2 Y^N\b\|^2_{\infty} + \b\|\D P^N\b\|^2_{\infty} \b).
\end{split}
\end{equation*}
The result follows by rescaling
\begin{alignat*}{3}
\b\|\D^3 Y^N\b\|_{\infty} &= N^{-2} \b\|\D^3 Y\b\|_{\infty}, \quad
&\b\|\D^2 Y^N\b\|^2_{\infty} &= N^{-2}\b\|\D^2 Y\b\|^2_{\infty}, \\
\b\|\D^2 P^N\b\|_{\infty} &= N^{-2}\b\|\D^2 P\b\|_{\infty}, \mbox{\quad and \quad}
&\b\|\D P^N\b\|^2_{\infty} &=N^{-2} \b\|\D P\b\|^2_{\infty}.
\end{alignat*}
\end{proof}

\begin{remark}
  Suppose that we are in the multi-species $2$-lattice case, and must
  use the symmetry $\simm$ instead of $\neg$. In this case, we have
  \begin{align*}
    e(\simm\ml{\rho}) &= \smfrac12 \D_{-\rho}^2 Y^N(\xi)
    +\b(\alpha - \smfrac12\b) \D_{-\rho} P^N(\xi)  \\
    &=\smfrac12 \D_{\rho}^2 Y^N(\xi)+\b(\smfrac12 - \alpha \b) \D_{\rho} P^N(\xi)
    \neq e(\ml{\rho}),
  \end{align*}
  whenever $\alpha = \beta$. Thus, in our above proof the first order
  terms do not cancel and the Cauchy--Born model is only first order
  accurate in this case. We will show in the next section how a new
  strain energy density can be constructed that is second order
  accurate in this case.
\end{remark}

\section{A second order accurate stored energy density for pair
  interactions}
\label{sec:pair}
In this section, we show how the symmetry $\sim$ for
multi-lattice pair interactions can be exploited to derive a
``simple'' second order accurate stored energy density, which does
{\em not} coincide with the Cauchy--Born strain energy density.
To maintain a notation consistent with the rest of the paper,
we state all results only for $2$-lattices, but stress that they can be immediately
extended to general multi-lattices composed of more than two Bravais lattices.

Suppose we are again given $1$-periodic deformation and shift fields
$(Y, P)$. Let $\b(Y^N, P^N\b)$ and $y^N = \b(y_0^N, y_1^N\b)$ be defined as
in the classical Cauchy--Born setting in \eqref{eq:definition of yN}.

To exploit the symmetry $\simm$ defined in~\eqref{eq:defn_sim},
we define a different multi-lattice directional derivative.
If, for the moment, we interpret $y^N$ as a continuum field,
then we may define
\begin{displaymath}
  \D_{\ml\rho} y^N := \smfrac12 \D_\rho y_\alpha^N + \smfrac12 \D_\rho
  y^N_\beta + (y_\beta^N -y_\alpha^N ).
\end{displaymath}
We observe that $\D_{\sim\ml\rho} = - \D_{\ml\rho}$ for this definition.
Written in terms of $(Y^N, P^N)$, $\D_{\ml\rho}$ becomes
\begin{equation}
  \label{eq:pair:defn_mldirder_pair}
  \D_{\ml\rho}(Y^N,P^N) := \D_\rho Y^N + (\beta-\alpha) P^N +
  \smfrac{\alpha+\beta}{2} \D_\rho P^N.
\end{equation}

Note that, in contrast to the previous multi-lattice directional
derivative~\eqref{eq:2latt_dirder},
this is not a scale-invariant field.
Indeed, defining, $\eps := 1/N$ and rewriting the directional derivative
in terms of $(Y, P)$, we arrive at
\begin{displaymath}
  \D_{\ml\rho}^\eps (Y, P) :=\D_{\ml\rho} \b(Y^N, P^N\b) = \D_\rho Y + (\beta-\alpha) P +
  \smfrac{\alpha+\beta}{2} \eps \D_\rho P,
\end{displaymath}
which gives rise to the new stored energy density
\begin{displaymath}
  W_\eps(\D Y, P) := V\b( \D^\eps_\Rg (Y, P) \b).
\end{displaymath}

\begin{remark}
  The fact that $W_\eps$ employs shift gradients but not strain
  gradients makes the resulting model straightforward to implement
  using $C^0$-conforming discretizations. By contrast, the
  higher-order models derived, e.g., in \cite{ArGr:2005} require also
  higher-order conforming discretizations.
\end{remark}

The symmetry $\D_{\sim\ml\rho} = - \D_{\ml\rho}$ was the first key
ingredient in the proof of of second order consistency of the
single-species $2$-lattice strain energy density.  We now generalize
the second key ingredient, the symmetry of partial derivatives of $V$.

\begin{lemma}
  Suppose that $V$ is point symmetric~\eqref{eq:invsymm_pair} with
  respect to the permutation operator $\simm$.  Let $\D_\ml{\rho}$ be
  defined by \eqref{eq:2latt_dirder}, and let $V_\ml{\rho} (\ml{g}) :=
  \frac{\partial}{\partial g_\ml{\rho}} V(\ml{g})$.  Then, for any
  deformation and shift fields $Y$ and $P$, we have
  \begin{equation}\label{eq:pair-2latt symmetry of partial derivatives}
      V_{\ml{\rho}} \b(\D_\Rg (Y^N,P^N) (x)\b)
      = -V_{\simm \ml{\rho}}\b(\D_\Rg (Y^N,P^N) (x)\b).
  \end{equation}
\end{lemma}
\begin{proof}
  The proof closely follows the proof of Lemma
  \ref{lemma:2latt symmetry of partial derivatives}.
\end{proof}

\begin{theorem}[Second order convergence for $2$-lattices]
  \label{thm:pair:2d order sstl}
  Let $Y \in C^3(\R^d; \R^d)$ be a $1$-periodic deformation field and
  $P \in C^2(\R^d; \R^d)$ a $1$-periodic shift field.  Let $y^N$ be
  the scaled atomistic configuration defined by \eqref{eq:definition
    of yN}; then
  \begin{equation*}
    \begin{split}
      N^{-d} \left | \Ea \b(y^N\b) - \Ec \b(Y^N,P^N \b) \right | \leq
      C N^{-2} \b \{ \, &\|\D^3 Y\|_{\infty} + \|\D^3 P\|_{\infty} \\
      & \quad+ \|\D^2 P\|_{\infty}  + \|\D^2 Y\|^2_{\infty} + \|\D P\|^2_{\infty} \b \}.
    \end{split}
  \end{equation*}
  where the constant $C$ is a function of $\Rg$ and
  $\|V\|_{C^2}$.
\end{theorem}

\begin{proof}
  The proof of the result is the same as the proof of Theorem
  \ref{thm:2d order sstl} up to \eqref{eq:thm cb 2latt 2},
  except that we need $P\in C^3$ for the quadrature estimate,
  since $\D P$ now enters the definition of the multi-lattice directional derivative.
  A computation analogous to \eqref{eq:thm cb 2latt 3} yields
  \begin{align*}
    D_{\ml\rho} y^N - \D_{\ml\rho}(Y^N,P^N) &=
    \smfrac{\beta-\alpha}{2} \D_\rho P^N + \smfrac{1}{2} \D_\rho^2 Y^N
    \\
    & \qquad + O\b( \| \D^3 Y^N \|_{\infty} + \| \D^2 P^N \|_\infty \b).
  \end{align*}
  Continuing as in the proof of Theorem \ref{thm:2d order sstl},
   we obtain
  \begin{align}
    V\b(D_\Rg y^N\b) - V\b(\D_\Rg \b(Y^N,P^N\b)\b)
    &= \sum_{\ml{\rho}\in \Rg} V_\ml{\rho} \cdot e (\ml{\rho})
    \nonumber \\
    &+O \b(\b\|\D^3 Y^N \b\|_\infty + \b\|\D^2 P^N \b\|_\infty \b),
    \label{eq:pair:thm cb 2latt 4}
  \end{align}
  where
  \begin{equation*}
    e(\ml{\rho}) := \smfrac{\beta-\alpha}{2} \D_\rho P^N + \smfrac{1}{2} \D_\rho^2 Y^N.
  \end{equation*}
  It is straightforward to see that $e(\simm\ml\rho) =
  e(\ml\rho)$, and hence the rest of the proof follows Theorem~\ref{thm:2d order sstl}.
\end{proof}

\section{Conclusion}

We have identified two new second order continuum models for multi-lattices.
Our approach is based on the identification of symmetries similar
to the point symmetry of Bravais lattices.
We then extend the standard proof of second order accuracy of
the Cauchy--Born rule for Bravais lattices to derive second order models for multi-lattices.

For single-species $2$-lattices,
we show that the classical Cauchy--Born model is of second order in the continuum limit,
\emph{provided} that the atomistic and continuum kinematic variables are related in a new way.
If the classical relationship between the variables is adopted,
then the Cauchy--Born model is only first order.
We also give a new stored energy density for a general multi-lattice modeled using pair interactions,
and we show that this energy is second order accurate, as well.

These results are being used to develop accurate atomistic/continuum couplings.
We show in~\cite{VKL:2011} that the interfacial error of a coupling can be dramatically
reduced when a point symmetric site energy is used,
and we make similar applications of the results of this paper in~\cite{OrVK:blend2}.
We also remark that our methods achieve second order accuracy without requiring the use
of $C^1$-conforming elements,
and so our models are easier to implement than the second order models of~\cite{ArGr:2005}.
\bibliographystyle{abbrv}

\end{document}

%% file: notation.tex
\newcommand{\smfrac}[2]{{\textstyle \frac{#1}{#2}}}

\def\XXint#1#2#3{{\setbox0=\hbox{$#1{#2#3}{\int}$ }
\vcenter{\hbox{$#2#3$ }}\kern-.6\wd0}}

\def\b{\big}
\def\B{\Big}

\def\R{\mathbb{R}}
\def\N{\mathbb{N}}
\def\Z{\mathbb{Z}}

\def\dx{\,{\rm d}x}

\def\<{\langle}
\def\>{\rangle}

\def\D{\nabla}

\def\eps{\varepsilon}

\def\a{{\rm a}}
\def\c{{\rm c}}

\def\Rg{\mathcal{R}}

\def\E{\mathscr{E}}
\def\Ea{\E^\a}
\def\Ec{\E^\c}



%% file: notation_multi.tex
\newcommand{\Y}{\mathscr{Y}}

\newcommand{\ignore}[1]{}

\newcommand{\Real}{\mathbb{R}}